\documentclass[11pt]{amsart}

\usepackage{amsmath,latexsym,amssymb,verbatim}
\usepackage{graphicx}
\usepackage{psfrag}
\usepackage{color}
\usepackage{tikz}
\usepackage{pgfplots}
\pgfplotsset{compat=1.17}
\usepackage{pstricks-add}
\usepackage{pst-func}

\newtheorem{lema}{Lemma}[section]
\newtheorem{theo}[lema]{Theorem}
\newtheorem{prop}[lema]{Proposition}

\theoremstyle{definition}

\theoremstyle{remark}
\newtheorem{rema}[lema]{Remark}

\definecolor{rojo}{rgb}{0.75,0,0}
\definecolor{oro}{rgb}{0.85, 0.65, 0.13}
\definecolor{azul}{rgb}{0.05, 0.05, 0.93}

\title[Hilbert number for piecewise nonautonomous equations]{Hilbert number for a family of piecewise nonautonomous equations}
\author{J.L. Bravo, M. Fern\'{a}ndez, and I. Ojeda}

\address{J.L. Bravo, Departamento de Matem\'{a}ticas,
Universidad de Extremadura, 06006 Badajoz, Spain}
\email{trinidad@unex.es}

\address{M. Fern\'{a}ndez, Departamento de Matem\'{a}ticas, Universidad de Extremadura, 06006 Badajoz, Spain}
\email{ghierro@unex.es}

\address{I. Ojeda, Departamento de Matem\'{a}ticas, Universidad de Extremadura, 06006 Badajoz, Spain }
\email{ojedamc@unex.es}


\begin{document}
 
\begin{abstract}
For family $x'=(a_0+a_1\cos t+a_2 \sin t)|x|+b_0+b_1 \cos t+b_2 \sin t$, we solve three basic problems related with its dynamics.  First, we characterize when it has a center (Poincar\'e center focus problem). Second, we show that each equation has a finite number of limit cycles (finiteness problem), and finally we give a uniform upper bound for the number of limit cycles (Hilbert problem 16).
\end{abstract}

\maketitle

\section{Introduction and Main Results}

Consider the scalar piecewise equation
\begin{equation}\label{eq:main}
x'= h(t,x)=
\begin{cases}
f(t,x),\quad \text{if }x\geq 0,\\
g(t,x),\quad \text{if }x < 0,\\
\end{cases}
\end{equation}
where $f(t,x)$ and $g(t,x)$ are  continuous and locally Lipschitz continuous with respect to $x$ in $\mathbb{R}^2$ and $f(t,0)=g(t,0)$. Under these hypotheses, there is a unique maximal solution that satisfies each of the initial conditions.

Let $u(t,\tau,x)$ be the maximal solution of \eqref{eq:main} determined by $u(\tau,\tau,x)=x$. Let $f,g \in\mathcal{C}^1(\mathbb{R}^2)$ and $(\tau,x)\in\mathbb{R}^2$. Suppose the set of zeroes of $f(t,0)=g(t,0)$ does not contain any interval, then there exists
$$
u_x(t,\tau,x)= \exp\left(\int_{\tau}^t h_x(s,u(s,\tau,x))\,ds \right)
$$
and it is  continuous with respect to $t$ and $x$ (see \cite[Appendix A]{BFT}). In this paper, subscripts in a function indicate its partial derivatives.

Now assume that $f$ and $g$ are $T-$periodic functions with respect to $t$. We say that $u(t,\tau,x)$ is closed or periodic if $u(T,\tau, x)=x$. Let $u(t,\tau, x)$ be closed. It is singular or multiple if $u_x(T,\tau,x)=1$, otherwise it is simple or hyperbolic.  Isolated periodic solutions are also called limit cycles. By analogy with the planar case and for simplicity, a continuum of periodic solutions (periodic annulus) is called a center of \eqref{eq:main}. A more precise definition is given in subsection~\ref{ss:center}.

Several authors have studied the case in which $f$ and $g$ are linear in $x$, i.e.  $f(t,x)=a^+(t) x +b(t)$ and $g(t,x)=a^-(t) x +b(t)$, where $a^+(t),a^-(t),b(t)$ are real continuous functions defined on $\mathbb{R}$. Then, since $f$ and $g$ are globally Lipschitz continuous with respect to $x$, the solutions are defined in $\mathbb{R}$.

In \cite{CFPT}, from a three-dimensional piecewise linear system with two zones, the following reduced one-dimensional $2\pi-$periodic differential equation is obtained
\begin{equation*} 
x'=\begin{cases}
a^+x+b(t),\text{ if } x\geq0,\\
a^-x+b(t),\text{ if } x<0,
\end{cases}
\end{equation*}
where $a^+,a^-\in\mathbb{R}$, $b(t)=b_0+b_1 \cos(t)$, $b_0\in\{0,1\}$, and $b_1\geq0$. The analysis of the limit cycles of this equation determines the dynamics of the three-dimensional piecewise linear system.

In \cite{JH} the authors obtain upper bounds for the number of limit cycles of 
\begin{equation*}
x'=\begin{cases}
a_1(t)x^m +b(t), \text{ if } x \geq0,\\
a_2(t)x^m +b(t), \text{ if } x <0,
\end{cases}
\end{equation*}
where $m$ is a positive integer and $a_1(t),a_2(t)$ and $b(t)$ are $2\pi-$periodic functions such that $a_1(t)$ and $a_2(t)$ do not change sign.

In \cite{GZ}, an upper bound for the number of limit cycles of equation 
\begin{equation}\label{eq:linearmain}
x'=a(t)|x|+b(t),
\end{equation}
is obtained by analyzing the displacement application, where $a(t)$ and $b(t)$ are real, $1-$periodic and $\mathcal{C}^1$ functions such that one of them does not change sign.

In contrast, when $a(t)$ or $b(t)$ changes sign, there is no uniform upper bound for the number of limit cycles of \eqref{eq:main}; indeed, in \cite{CGP} (see also \cite{BFT}) it is shown that for any natural number $k\geq2$ and $\epsilon$ small enough, the differential equation $x'=\epsilon \cos(kt)|x|+\ sin ( t )$ has at least $k-2$ limit cycles. However, if $f(t,x)$ and $g(t,x)$ are $T-$periodic analytic functions, the following finiteness theorem holds.

\begin{theo}[\cite{BFT}] \label{theo:finite}
Let $f$ and $g$ be real $T-$periodic analytic functions such that  $f(t,0)=g(t,0)$ has $n$ simple zeroes in $[0,T]$. Then there  exist $x_1<\ldots<x_r$, where $1\leq r\leq n+2$, such that, if $u(t)$ is a periodic solution that changes sign, then $u(0)\in [x_1,x_r]$. Moreover, for each $1\leq k< r$ one of the following statements holds:
\begin{enumerate}
\item There is a finite number of values $x\in[x_k,x_{k+1}]$ such that  $u(t,0,x)$ is periodic.
\item For every $x\in[x_k,x_{k+1}]$,  $u(t,0,x)$ is periodic.
\end{enumerate}
\end{theo}

In this paper we solve three fundamental problems related to the dynamics of \eqref{eq:linearmain} when $a(t)$ and $b(t)$ are linear trigonometric polynomials. First, we characterize when \eqref{eq:linearmain} has a center (Poincar\'e center-focus problem). Second, we prove that each equation has a finite number of limit cycles (finiteness problem), and finally we set a uniform upper bound for the number of limit cycles (Hilbert's problem 16).

\section{Linear trigonometric coefficients}

Consider the Abel piecewise linear equation
\begin{equation}\label{eq:abs}
x'=a(t)|x| + b(t),
\end{equation}
where $a(t)=a_0 + a_1\cos(t)+a_2 \sin(t)$ and $b(t)= b_0 + b_1\cos(t)+b_2 \sin(t)$. 

Since periodic solutions with constant sign are determined by a linear differential equation, the key point is to study periodic solutions that change sign, that is, those solutions $u(t)$ such that there exist $t_1,t_2$ such that $u(t_1)u(t_2)<0$. In particular, $b(t)$ must change sign, otherwise $x=0$ is a lower or an upper solution, thus avoiding the existence of periodic solutions that change sign. By a translation of time, it is not restrictive to assume $b(0)=0$ and $b(t)>0$ for $t>0$ close to $0$. Moreover, since $b(t)$ is linear, $b(t)=0$ has a unique solution in $(0,2\pi)$, which we denote $\bar t$.

In consequence, if $u(t)$ is a periodic solution of~\eqref{eq:abs} that changes sign, then there exist $0<t_1<\bar t<t_2<2\pi$ such that $u(t_1)=u(t_2)=0$. Moreover, $u(t)<0$ if $t\in (0,t_1) $, $u(t)>0$ if $t\in (t_1,t_2)$ and $u(t)<0$ if $t\in (t_2,t_1+2 \pi)$. Since $u(t)$ is a solution of the IVP $x'=a(t)x+b(t)$, $x(t_1)=0$ for $t\in (t_1,t_2)$, and of $x'=-a(t)x+b(t)$, $x(t_2)=0$ for $t\in (t_2,t_1+2\pi)$, and $u(t_1+2\pi)=0$, we have 
\begin{equation}\label{eq:zero}
\begin{split}
\int_{t_1}^{t_2} b(t)\exp\left(\int_t^{t_2} a(s)\,ds\right)\, dt =0,\\
\int_{t_2}^{t_1+2\pi} b(t)\exp\left(\int_t^{t_1+2\pi} -a(s)\,ds\right)\,dt =0,
\end{split}
\end{equation}
and conversely, if there exist $0<t_1<\bar t<t_2<2\pi$ such that \eqref{eq:zero} holds,
then there exists a periodic solution $u(t)$ of~\eqref{eq:main} determined by $u(t_1)=0$. This allows to reduce the problem of finding periodic solutions to the problem of finding the points where two Poincar\'e half-maps intersect. Specifically, let 
\[T^+\colon (0,\bar t)\to (\bar t,2\pi)\quad \text{and}\ 
\quad T^-\colon (2\pi,2\pi+\bar t)\to (\bar t,2\pi),\]
where $T^+(t_1)$ is the first time where the solution $u(t)$ of~\eqref{eq:abs} determined by $u(t_1)=0$ is zero, that is, 
\begin{equation} \label{eq:T+}
T^+(t_1)=\inf\{t_2\colon t_2>t_1,\  u(t_2)=0\},
\end{equation}
and $T^-(t_1)$ is  determined by
\begin{equation}\label{eq:T-}
T^-(t_1)=\sup\{t_2<t_1+2\pi \colon t_2>t_1,\  u(t_2)=0\},
\end{equation}
where $u(t)$ is, now, determined by $u(2\pi+t_1)=0$.

\begin{prop}\label{prop:1}
The periodic solutions of \eqref{eq:abs} that change sign are in a one-to-one correspondence with the solutions of $T^+(t)=T^-(t)$, $t\in(0,\bar t)$.
\end{prop}

\begin{proof}
Assume $u(t)$ is a periodic solution of \eqref{eq:abs} that changes sign. 

If $u(0)=u(2\pi)\geq 0$, as $b(t)>0$ for $t\in(0,\bar t)$, then $u(t)\geq 0$ for $t\in[0,\bar t]$, and, as $b(t)<0$ for $t\in (\bar t,2\pi)$, then $u(t)\geq 0$ for $t\in [\bar t, 2\pi]$, so it has definite sign, in contradiction with the assumption. 

Therefore, $u(0)=u(2\pi)<0$. Since $b(t)$ has a single zero at $(0,2\pi)$, $u(t)$ has at most one sign change at both $(0,\bar t)$ and $(\bar t,2\pi )$. So, as $u(t)$ has indefinite sign, there exist $0<t_1<\bar t<t_2<2\pi$ such that $u(t_1)=u(t_2)=u(t_1+2\pi)=0$. Therefore, \eqref{eq:zero} holds and, consequently, $T^+(t_1)=T^-(t_1)$.

Conversely, suppose $t\in(0,\bar t)$ is such that $T^+(t)=T^-(t)$. Note that if $T^+(t)$ is well defined, by the sign of $b(t)$, we have that \[T^+(t)=T^-(t)\in (\bar t,2\pi).\]

Let $u(t)$ be the solution of~\eqref{eq:abs} determined by $u(t)=0$. As $u(t)$ can change sign at most once in each interval where $b(t)$ has definite sign, $u(s)>0$ for $t<s<\bar t$, and, by \eqref{eq:T+}, $u(T^+(t))=0$. And, since $T^+(t)=T^-(t)$, then $u(T^-(t))=0$. Therefore, $u(t)<0$ for $T^-(t)<s<2\pi$, and, by \eqref{eq:T-}, $u(t+2\pi)=0$, so $u(t)$ is a periodic solution.
\end{proof}

Note that the Poincar\'e half-maps are defined implicitly by 
\begin{equation}\label{eq:imp2}
\begin{split}
\int_{t}^{T^+(t)} b(s) \exp\left(\int_s^{T^+(t)}a(r)\,dr \right) \, ds &= 0,\\  \int_{T^-(t)}^{t+2\pi} b(s) \exp\left(-\int_s^{t+2\pi}a(r)\,dr \right) \, ds &= 0.
\end{split}
\end{equation}

\begin{lema}\label{lema:Ts}
The functions $T^+(t)$ and $T^-(t)$ are solutions of the differential equations
\begin{align}\label{eq:f}
b(x)x'- b(t) \exp\left(\int_t^x a(s)\,ds\right)&=0,\\
\label{eq:g} b(x)x'-b(t) \exp\left( \int_{t+2\pi}^x -a(s)\,ds\right) &=0,
\end{align}
respectively.
\end{lema}

\begin{proof}
It is enough to observe that deriving \eqref{eq:imp2} with respect to $t$, we obtain
\begin{gather}\label{deq:T}
b(T^+(t)) \frac{d T^+}{dt}(t)-b(t)\exp\left(\int_t^{T^+(t)}a(s)\,ds   \right)=0,\\
b(T^-(t)) \frac{d T^-}{dt}(t) -b(t)\exp\left(-\int_{t+2\pi}^{T^-(t)}a(s)\,ds   \right)=0.
\end{gather}
\end{proof}

\subsection{The center problem}\label{ss:center}
Let $u(t,x)$ denote the solution of~\eqref{eq:abs} determined by $u(0,x)=x$. Given $x\in\mathbb{R}$  define the displacement application 
\[d(x) := u(2\pi,x)-x.\] Note that the solution $u(\cdot,x)$ is $2\pi-$periodic if and only if $d(x)=0$. Equation~\eqref{eq:abs}  has a center if $d(x)=0$ for every $x$ in an interval. We say that \eqref{eq:abs} has a global center if $d$ is identically zero.

A linear center is defined as a center where the periodic solutions of the continuum have definite sign.

For $x>0$ large enough, the solution $u(\cdot,x)$ is positive, so the displacement application is
\[
d(x)=(A-1)x+B,
\]
where
\begin{equation}\label{eq:AB}
A=\exp\left(\int_0^{2\pi}a(r)\,dr\right),\quad B=\int_0^{2\pi}b(s)\exp\left(\int_s^{2\pi}a(r)\,dr \right)\,ds.
\end{equation}
Analogously, for $x<0$ small enough  $d(x)=(\bar A-1)x+\bar B$, where $\bar A,\bar B,$ are defined analogously.
Therefore, for $|x|$ large enough, the equation is linear, so it has a center
if $A=1$, $B=0$, for $x\gg 0$, and if $\bar A=1$, $\bar B=0$ for $x\ll 0$.

If $b(t)\equiv 0$, the non-zero solutions are either always positive or negative, so the possible centers are linear (or global) and determined by conditions above. The general case is covered in the next result.

\begin{theo}
Equation \eqref{eq:abs} has linear centers or a global center. Moreover, it has a global center if and only if $a_0=0$ and $a(t)$ and $b(t)$ are proportional.
\end{theo}

\begin{proof}
If $b(t)$ is identically zero, in particular it is proportional to $a(t)$. Since $u(t)\equiv0$ is a solution, there are no periodic solutions that change sign. Therefore, \eqref{eq:abs} has only linear centers. Moreover, by comment above, \eqref{eq:abs} has a global center if and only if $A=\bar A =1$, which is equivalent to $a_0=0$.

If $b(t)$ does not change sign and it is not identically zero, then $x=0$ is either a lower or an upper solution, so there are no periodic solutions that change sign. That is, there are only linear centers and they are characterized by $A=1,B=0$ when the periodic solutions are positive and $\bar A=1,\bar B=0$ when the periodic solutions are negative.

Assume that $b(t)$ changes sign. First we prove that if \eqref{eq:abs} has a center that is not linear, that is, if there exists a continuum of periodic solutions that change sign, then \eqref{eq:abs} has a global center. Recall that we can assume that $b(0)=b(\bar t)=0$ are the unique zeroes of $b(t)$ in $[0,2\pi)$, and that $b(t)>0$ for $t\in (0,\bar t)$ and $b(t)<0$ for $t\in (\bar t,2\pi)$. Then there is a continuum of periodic solutions crossing $x=0$ twice, one in $(0,\bar t)$ and one in $(\bar t,2\pi)$. By analyticity, these crossing must cover all $(0,\bar t)$ and $(\bar t,2\pi)$. In particular, the value at $t=0$ of these solutions is an interval of the form $(\bar x,0)$, where $u(t,\bar x)$ and $u(t,0)$ do not change sign, and the displacement application $d$ is identically zero in the interval $(\bar x,0)$ and linear in $(-\infty,\bar x ]\cup [0,+\infty)$. However, since $d$ is differentiable, we conclude that $d$ is identically zero, and thus the center is global.

Next we show that the global centers are characterized by the fact that $a_0=0$ and that the functions $a(t)$ and $b(t)$ are proportional. Assume that \eqref{eq:abs} has a global center. As for $x\gg 0$, $d(x)= (A-1)x+B$, then $A=1$, so $a_0=0$. Also, since all solutions are periodic, solutions that change sign are, then $T(t):=T^+(t)=T^-(t)$ for every $t\in[0,\bar t]$. 
Therefore, the 
vector fields of~\eqref{eq:f} and \eqref{eq:g}, coincide, so 
\[
\int_t^{T(t)} a(s)\,ds = \int_{t+2\pi}^{T(t)} -a(s)\,ds.
\]
for every $t\in (0,\bar t)$. But, as $a_0=0$,
\[
\begin{split}
0 = 
\int_t^{T(t)} a(s)\,ds + \int_{t+2\pi}^{T(t)} a(s)\,ds
&= 2 \int_t^{T(t)} a(s)\,ds + 
\int_{t+2\pi}^t a(s)\,ds \\ 
& = 2 \int_t^{T(t)} a(s)\,ds.
\end{split}
\]
Considering that for some constants $c\neq 0$ and $t_a\in (0 ,2 \pi)$, $a(t) = a_1 \cos(t)+a_2\sin(t)=c\sin(t-t_a)$, we have that $\cos(t-t_a)=\cos(T(t)-t_a)$, which implies that $T(t)=-t+2\pi$ because $0\leq t<\bar t< T(t)\leq 2\pi$, $T(t)$ is strictly decreasing and $T(0)=T(2\pi)$. If we now derive the expression $\int_t^{-t+2\pi}a(s)\,ds=0$ with respect to $t$, we obtain $a(t)+a(-t+2\pi)=0$, implying $a(t)=a_2 \sin(t)$. Moreover, $T(t) = 2\pi-t$ also implies that $\bar t=\pi$, so $b(t)$ is proportional to $\sin(t)$, and $a_2 b(t)  = a_2 \sin(t) = a(t)$

Conversely, if $a_0=0$ and $b(t)=\lambda a(t) \not\equiv 0$, then there exist $c>0$ and $\gamma \in[0,2\pi)$ such that $a(t)=c\sin(t-\gamma)$. By the change of variables $t\to t+\gamma,\ x\to x/c$, equation \eqref{eq:abs} transforms into $x'=\sin(t)|x|+\lambda \sin(t)$ which has a center because $u(t)=u(-t)$ for every solution $u(t)$. 
\end{proof}

\subsection{The finiteness problem}

The finiteness problem is essentially solved in Theorem~\ref{theo:finite} of \cite{BFT}. Thus, we only show here that \eqref{eq:abs} satisfies the hypotheses of that theorem.

\begin{theo}
Equation \ref{eq:abs}  has a finite number of limit cycles.
\end{theo}
\begin{proof}
If the zeroes of $b(t)$ are simple, then the result follows from Theorem~\ref{theo:finite}. Now, as $b(t)$ is linear trigonometric, if $b(t)$ has multiple zero, then it does not change sign, so $x=0$ is a lower or an upper solution. In particular, no periodic solution crosses $x=0$, and there is at most one positive and one negative periodic solution because the equation is linear in these regions.
\end{proof}

\subsection{Hilbert's 16th problem}

When $a(t)$ or $b(t)$ does not change sign, an uniform upper bound is $2$ (see \cite{GZ,Ma}), so we assume that $a(t)$ and $b(t)$ change sign. 

If we denote \[F(t,x,a_0,b_0):=\big(a_0+a_1 \cos(t)+a_2 \sin(t)\big)|x|+b_0\big(1-\cos(t)\big)+\sin(t),\] then $\frac{\partial F}{\partial a_0}=|x|\geq0$ and $\frac{\partial F}{\partial b_0}=1-\cos(t)\geq0$, being strict inequalities for $x\neq0$ and $t\neq 2k\pi$, $k\in \mathbb{Z}$, respectively. Thus, since the vector field defined by \eqref{eq:abs} is increasing with respect to $a_0$ and $b_0$, the same is true for the displacement application, that is, if $d(x,a_0)$ (resp. $d(x,b_0)$) denotes the displacement application in terms of $a_0$ (resp. $b_0$), then $d$ is strictly increasing with respect to $a_0$ (resp. $b_0$).

\begin{prop}
Assume that~\eqref{eq:abs} has at least $n$ limit cycles for a certain value of $a_0$ (resp. $b_0$). Then there exist nearby values of $a_0$ (resp. $b_0$) such that~\eqref{eq:abs} has at least $n$ limit cycles. 
\end{prop}

\begin{proof}
Assume $\bar x$ is a zero of $d$ such that $d$ changes sign. Then $d$ changes sign for nearby values of $a_0$ for any sufficiently small neighborhood of $\bar x$.

Assume $\bar x$ is a zero of $d$ such that $d$ does not change sign for $a_0=\bar a_0$; for instance, $d(x,\bar a_0)\geq 0$ for $x$ in a neighborhood of $\bar x$, being the strict inequality for $x\neq \bar x$. Then for $a_0<\bar a_0$ sufficiently small, $d$ has at least two zeroes close to $\bar x$. 

Now, by conveniently selecting $a_0$, we get the desired result.  
\end{proof}

From the previous proposition, to obtain a uniform upper bound for the limit cycles of \eqref{eq:abs}, we can exclude any finite number of values of $a_0$ and $b_0$. In particular, in the following we assume that $a_0 \neq 0$.

When $b(t)$ changes sign, it has two simple roots in $[0,2\pi)$. By a translation of time, it is not restrictive to assume that $0$ is one of them and that $b'(0)>0$. Moreover, by rescaling $x$ we can assume that the coefficient of $\sin(t)$ is one, so we assume that
\[
b(t) = \sin(t) + b_0(1-\cos(t)).
\] 
Moreover, we have that $b(t)>0$ for $t\in(0,\bar t)$ and $b(t)<0$ for $t\in (\bar t,2\pi)$, where $\bar t$ is the zero of $b(t)$ in $(0,2\pi)$. By Proposition~\ref{prop:1}, there is a one-to-one correspondence between the periodic solutions of \eqref{eq:abs} that change sign and the solutions of $T^+(t)=T^-(t)$.

Unfortunately, the equation $T^+(t)=T^-(t)$ is difficult to handle, since $T^+$ and $T^-$ are implicitly defined functions. In the following steps, the problem of bounding the number of solutions to this equation is reduced to bounding the number of solutions to simpler equations. For this, we first consider the function
\begin{equation*}
h(t,x) = \int_t^x a(s)\,ds + \int_{t+2\pi}^x a(s)\,ds. 
\end{equation*}	

\begin{prop}\label{prop:der1}
Let $n-1$ be the maximum number of isolated points in the intersection of $h(t,x)=0$ with the graph of a solution of \eqref{eq:f} in $(0,\bar t)\times (\bar t,2\pi)$. Then the number of limit cycles of \eqref{eq:abs} that change sign is less than or equal to $n$.
\end{prop}

\begin{proof}
First, we recall that, by Lemma \ref{lema:Ts}, $T^+(t)$ and $T^-(t)$ are solutions of \eqref{eq:f} and \eqref{eq:g} respectively. Consider the vector fields defined by \eqref{eq:f} and \eqref{eq:g}. 
The set of points  $(t,x) \in (0,\bar t) \times (\bar t,2\pi)$ where the two vector fields are proportional are the solutions of 
\[
-\frac{b(t)}{b(x)}\exp\left(\int_t^x a(s)\,ds\right) 
=-\frac{b(t)}{b(x)}\exp\left(\int_{t+2\pi}^x -a(s)\,ds\right). 
\]
This equation is equivalent to $h(t,x)=0.$

By Proposition~\ref{prop:1}, the number of limit cycles of \eqref{eq:abs} that change sign is the number of solutions of $T^+(t)=T^-(t)$, which is bounded by the number of intersections of  the graph of a solution of \eqref{eq:f} with the graph of a solution of \eqref{eq:g}, both restricted to $(0,\bar t)\times (\bar t,2\pi)$. Let $u(t)$ and $v(t)$ be solutions of \eqref{eq:f} and \eqref{eq:g}, respectively. If the graphs of $x=u(t)$ and $x=v(t)$ intersect in two consecutive points of $u(t)$, then both vector fields are proportional in an intermediate value, that is, the graph of $x=u(t)$ and $h(t,x)=0$ intersect at one point. Therefore, if $n-1$ is the number of intersections of the graph of $x=u(t)$ and $h(t,x)=0$, then the number of intersections of the graphs of $x=u(t)$ and $x=v(t)$ is at most $n$. 
\end{proof}

To control the number of intersections of solutions of \eqref{eq:f} or \eqref{eq:g} with $h(t,x)=0$ we need some additional properties of $h(t,x)=0$. Note that 
\begin{equation}
\begin{split}
h(t,x)=&   \int_t^x a(s)\,ds +
\int_{t+2\pi}^x a(s)\,ds \\=& 
2 \int_t^x a(s)\,ds 
+ \int_{t+2\pi}^t a(s)\,ds\\=&2 \int_t^x a(s)\,ds - 2\pi a_0.
\end{split}
\end{equation}

\begin{lema}\label{le:h0}
The graph of $h(t,x) = 0$ for $(t,x)\in R:=(0,\bar t)\times (\bar t, 2\pi)$ is, after a linear change of variables, the restriction to $R$ of the graph of a function.
\end{lema}

\begin{proof}
As $a_0 \neq 0$, then \[\frac{1}{a_0} h(t,x) = 2 \left( r_1 \sin(x)-r_2 \cos(x)+x-r_1 \sin(t) + r_2 \cos(t)-t\right) -2 \pi\] with $r_i = a_i/a_0,\ i = 1,2$. 

By the  change of variables:
\[
\left(\begin{array}{c} z_1 \\ z_2 \end{array}\right) = 
\left(\begin{array}{rr} 1 & 1 \\ -1 & 1  \end{array}\right) \left(\begin{array}{c} t \\ x \end{array}\right),
\]
$\frac{1}{a_0} h(t,x)$ can be written in the form
\[
\frac{1}{a_0} h(z_1,z_2)=4 \left(r_1 \cos(z_1/2) + r_2 \sin(z_1/2)\right) \sin(z_2/2) - 2\pi + 2z_2.
\]
Thus the graph of the equation $h(t,x)=0$, restricted to $0<t<\bar t$ and $\bar t<x<2\pi$, becomes in part of the (bigger) graph of 
\begin{equation}\label{ecu:hz}
\frac{\pi - z_2}{\sin(z_2/2)} = 2\left(r_1 \cos(z_1/2) + r_2 \sin(z_1/2)\right)
\end{equation}
restricted to $\bar t < z_1 < \bar t +  2\pi$ and $0 < z_2 < 2 \pi$. Now, since the left hand side of \eqref{ecu:hz} is a strictly decreasing function in $(0, 2\pi)$, we conclude that $z_2$ is a function of $z_1$.
\end{proof}

Now, by studying the tangencies of $T^+$ along the curve $h(t,x)=0$, it is possible to simplify the problem even further. Let us denote \[m(t,x):=a(t)b(x)-a(x)b(t)e^{\pi a_0}.\]

\begin{prop}\label{prop:der2}
Let $n-2$ be the number of isolated solutions of 
\[
h(t,x)= m(t,x)=0,
\]
where
$0<x-t<2\pi$ and $\bar t<t+x<\bar t+2\pi$.
Then the number of limit cycles of \eqref{eq:abs} that change sign is less than or equal to $n$.
\end{prop}

\begin{proof}
By Lemma~\ref{le:h0}, $h(t,x)=0$ is the graph of a function in the region $0<x-t<2\pi$ and $\bar t<t+x<\bar t+2\pi$. Then, between any two consecutive intersections of $h(t,x)=0$ and   the graph of a solution of \eqref{eq:f}, the vector field defined by \eqref{eq:f} must be tangent to the curve, that is, for some point of $h(t,x)=0$,
$$ l(t,x):=h_x(t,x) b(t) \exp\left(\int_t^x a(s)\,ds\right)+h_t(t,x)b(x)=0.$$
Moreover, since $h(t,x)=0$ is equivalent to $\int_t^x a(s)\,ds=\pi a_0$, the set of points such that $l(t,x)=h(t,x)=0$ is the set of points such that
\[
h(t,x) =a(t)b(x)-a(x)b(t)e^{\pi a_0}=0.
\]
Therefore, if the number of isolated intersections of the curve $h(t,x)= 0$ with $m(t,x)=0$ is at most $n-2$, then the number of isolated intersections of the curve $h(t,x)=0$ with the graph of  a solution of \eqref{eq:f} in $(0,\bar t)\times (\bar t,2\pi)$ is at most $n-1$ and, by  Proposition~\ref{prop:der1}, we are done. 
\end{proof}

The following result guarantees finiteness and a (very thick) upper bound for the number of limit cycles. To obtain this bound, we use the following Khovanski\u{\i} theorem (\cite[Section 1.4]{fewnomials}), which gives an estimate of the number of solutions of a system of real exponential and trigonometric polynomials.

\begin{theo}{(Khovanski\u{\i})}\label{Th Kho}. Let $P_1 = \ldots = P_n = 0$ be a system of $n$ equations with $n$ real unknowns $x = x_1, \ldots, x_n$, where $P_i$ is a polynomial of degree $m_i$ in $n+k+2\rho$ real variables $x, y_1, \ldots, y_k,$ $u_1, \ldots, u_\rho, v_1, \ldots, v_\rho$, where $y_j = \operatorname{exp}\langle a_j, x \rangle,\ j = 1, \ldots, k$ and $u_q = \sin \langle b_q, x \rangle,$ $v_q = \cos \langle b_q, x \rangle,\ q = 1, \ldots, \rho$. Then the number of isolated real solutions of this system in the region bounded by the inequalities $\vert \langle b_q, x \rangle \vert < \pi/2,\ q = 1, \ldots, \rho$ is finite and less than
\[
m_1 \cdots m_n \left(\sum_{i=1}^n m_i + \rho + 1\right)^{\rho+k} 2^{\rho + (\rho + k)(\rho + k -1)/2}
\]
\end{theo}

\begin{theo}\label{theo:h16}
The number of limit cycles of \eqref{eq:abs} is less than or equal to $9834500 = 7^4 \cdot 2^{12}+2+2$.    
\end{theo}

\begin{proof}
First, note that there are at most two limit cycles that do not change sign, so we only need to limit the number of limit cycles that change sign. Also, as we noted earlier, we assume $a_0 \neq 0$.

Now, since \[\frac{1}{a_0} h(t,x) = 2 \left( r_1 \sin(x)-r_2 \cos(x)+x-r_1 \sin(t) + r_2 \cos(t)-t\right) -2 \pi\] and 
\begin{align*}
\frac{2}{a_0} m(t,x) = & a(t)b(x)-a(x)b(t)e^{\pi a_0} \\ = 
& (e^{\pi a_0}-1)(b_0 r_2 - r_1)\sin(x+t) \\
& + (e^{\pi a_0}-1)(b_0 b_1 + r_2)\cos(x+t) \\
& + (e^{\pi a_0}-1)(b_0 r_2 - r_1)\sin(x-t) \\
& + (e^{\pi a_0}-1)(b_0 r_1 - r_2)\cos(x-t) \\
& + 2 (b_0 r_2 - e^{\pi a_0})\sin(t) + 2 b_0 (r_1 + e^{\pi a_0}) \cos(t) \\ & - 2(e^{\pi a_0} b_0 r_2 - 1) \sin(x)
- 2b_0(e^{\pi a_0} r_1+1) \cos(x)
\\ & -2(e^{\pi a_0}-1) b_0
\end{align*}
with $r_i = a_i/a_0,\ i = 1,2$, are polynomials of degree $1$ in the variables $t, x, \sin(t),$ $\sin(x), \sin(x+t),  \sin(x-t), \cos(t), \cos(x), \cos(x+t),$ and $\cos(x-t)$. By the Khovanski\u{\i} theorem (Theorem \ref{Th Kho}), with $n=2, x_1 = t, x_2 = x, k=0,  m_1=m_2=1, b_1 = (1,0), b_2 = (0,1), b_3=(1,1)$ and $b_4=(-1,1)$, that is $\rho=4$, we have that the number  of isolated real solutions of the system $h(t,x) = m(t,x) = 0$ in the region \begin{align*} Q & = \{\vert \langle b_q, x \rangle \vert < \pi/2\ : q = 1, \ldots, 4\} \\ & = \{(x,t) \in \mathbb{R}^2 : \vert x+t \vert < \pi/2, \vert x-t \vert < \pi/2\}
\end{align*} is finite and less than $7^4 \cdot 2^{10}$. Now, since we need four translations of $Q$ to cover the region $(0,\pi) \times (\pi, 2\pi)$, we conclude that the number of isolated real solutions of the system $h(t,x) = m(t,x) = 0$ for $(t,x) \in (0,\pi) \times (\pi, 2\pi)$ is bounded by $4 \cdot 7^4 \cdot 2^{10} = 7^4 \cdot 2^{12}$. 

Finally, by Proposition \ref{prop:der2}, we conclude that the of limit cycles of \eqref{eq:abs} that change sign is less than or equal to $7^4 \cdot 2^{12} +2$.
\end{proof}

\subsection{Improving the upper bound}
The aim of this subsection is to improve the upper bound given by Theorem~\ref{theo:h16}, by performing a more detailed analysis of the system of equations $m(t,x) = h(t,x) =0$.

Note that $m(t,x)$ is a trigonometric polynomial, while $h(t,x)$ is the primitive of a trigonometric polynomial. Our goal is to reduce the problem of bounding the number of limit cycles to computing the number of intersections of two trigonometric polynomial curves, which can be done by Bezout's theorem. Then, we repeat the arguments above and study the tangencies of $h(t,x)$ along the branches of  $m(t,x)=0$, so the next step is to study the implicit equation $m(t,x)=0$. 

Recall that as in the previous subsection, we assume that  $a(t)$ and $b(t)$ change sign, and that $a_0$ and $b_0$ are generic. 

\begin{prop}\label{prop:3}
The number of connected components of $m(t,x)=0$ in $(0,2 \pi)\times (0,2\pi)$ is at most $4$.
\end{prop}

\begin{proof}
As $a_0,b_0$ are generic, we may assume that $a(t)$ and $b(t)$ have not common zeroes in $[0,2\pi]$. By the change of variables $t \to 2 \operatorname{atan}(z_1)+\pi$ and $x \to 2 \operatorname{atan}(z_2)+\pi$, $m(t,x)$ becomes \[\phi(z_1,z_2) := \frac{2\, f(z_1,z_2)}{(z_1^2+1)(z_2^2+1)},\] where $f(z_1,z_2)$ is a polynomial of degree $3$ which is irreducible over the complex by Lemma \ref{le:ceros_comunes}. 
According to \cite[p. 300]{KW}, applying Harnack's theorem, we obtain that the  real connected components of an affine algebraic algebraic curve of degree $d$ is less than or equal to $\max(g+1,g+N)$, where $N$ is the number of non-compact connected components of the affine algebraic curve and $g$ is the genus of the corresponding projective curve. Therefore, since $f(z_1,z_2) = 0$ defines a curve of degree $3$ in $\mathbb{R}^2$ we conclude that $g = \frac{(3-1)(3-2) )}{ 2} = 1$ and, according to Bezout's theorem, that $N$ is less than or equal to $3$. So, $f(z_1,z_2) = 0$ has at most $1+3 = 4$ connected components in $\mathbb{R}^2$. Therefore, $\phi(z_1,z_2)$ has at most four connected components. Now, since the above change of variables is a diffeomorphism from $\mathbb{R}^2$ to $(0,2 \pi) \times (0,2\pi)$; in particular, it is an homeomorphism and we conclude that $m(t,x)=0$ has at most four connected components in $(0,2 \pi) \times (0,2\pi)$. 
\end{proof}

Since we are assuming $a_0\neq 0$, dividing $a(t)$ by $a_0$, we can take $\bar a(t):=a(t)/a_0 = 1 + r_1 \cos(t) + r_2 \sin(t)$ and $c := e^{\pi a_0}$. Then $m(t,x)=0$ is equivalent to $\bar a(t)b(x)-\bar a(x)b(t) c=0$, which for $t,x$ such that $b(t)b(x)\neq 0$ can be rewritten as \[k(t)-k(x)c=0,\] where $k(t) = \bar a(t)/b(t)$. 

By Proposition~\ref{prop:der2}, 
to bound the number of limit cycles of~\eqref{eq:abs}, it suffices to compute the number of isolated solutions of 
$h(t,x)=0$, $m(t,x)=0$,
for 
$0<x-t<2\pi$ and $\bar t<t+x<\bar t+2\pi$. Note that in every connected component of $m(t,x)=0$, if we have two consecutive zeros of $h(t,x)=0$, then the gradients of $m(t,x)=0$ and of $h(t,x)=0$ must be proportional. Therefore, the number of isolated zeros of $m(t,x)=h(t,x)=0$
is bounded by the number of points of $m(t,x)=0$ where both gradients are proportionals, plus four by Proposition \ref{prop:3}. 

Next, we show that the points of $m(t,x)=0$ where both gradients are proportional are, after a suitable variable change, the real solutions of two algebraic equations, so that we can bound their number using Bezout's theorem.

Observe that the gradient of $m(t,x)=0$ is $(k'(t),-k'(x)c)$, where $k'(t) = -n(t)/b(t)^2$ with \[n(t) := ((b_0r_1+b_0)\sin(t)+(1-b_0r_2)\cos(t)+b_0r_2+r_1.\]

\begin{prop}\label{prop01}
The set of points $(t,x) \in (0,\pi)\times (\pi,2\pi)$ such that $m(t,x)=0$ and $(h_t,h_x)=(-2 \bar a(t),2 \bar a(x))$ is proportional to $(k'(t), -k'(x) c)$ is the set $\mathcal Z$ of solutions of 
\begin{equation}\label{ecu01}
\begin{cases}
& \bar a(t)b(x)-\bar a(x)b(t) c=0,\\
& n(t)b(x)^3-n(x)b(t)^3 c^2=0.
\end{cases}
\end{equation}
The number of isolated points of $\mathcal Z$ is less than or equal to $27$.
\end{prop}

\begin{proof}
A point $(t,x) \in (0,\pi)\times (\pi,2\pi)$ satisfies both that $m(t,x)=0$ and that $(h_t,h_x)=(-2 \bar a(t),2 \bar a(x))$ is proportional to $(k'(t), -k'(x) c)$ if and only if the matrix 
\[
\left(\begin{array}{ccc}  -\bar a(t) & b(t) c & k'(t) \\
\bar a(x) & -b(x) & -k'(x) c \end{array}\right)
\]
has rank two, that is, the $2 \times 2-$minors of the above matrix are equal to zero.

Now, by the change of variables $t \rightarrow 2 \operatorname{atan}(z_1)+\pi,\ x \rightarrow 2 \operatorname{atan}(z_2)+\pi$, we obtain that \eqref{ecu01} is equivalent to a polynomial system of equations determined by two polynomials; one of degree $3$ and another of degree $9$. Therefore, by Bezout's theorem (see \cite[Theorem 4.106]{ARAG} for more details) $\mathcal Z$ has at most 27 isolated points.
\end{proof}

\begin{rema}\label{rema:15}
Applying Gr\"obner basis techniques, one can show that cardinality of the set $\mathcal Z$ in Proposition \ref{prop01} is less than or equal to $15$. Indeed, using Singular \cite{Singular}, we can compute a Gr\"obner basis of the ideal $\mathfrak a$ of $\mathbb{C}[z_1,z_2]$ generated by the two polynomials that define the system equivalent to \eqref{ecu01} after the change of variable $t \rightarrow 2 \operatorname{atan}(z_1)+\pi,\ x \rightarrow 2 \operatorname{atan}(z_2)+\pi$, for generic $r_1, r_2, b_0$ and $c$.
The initial ideal of $\mathfrak{a}$ with respect to the graded reverse lexicographic monomial order on $\mathbb{C}[z_1,z_2]$ is generated by $\{z_1^6,z_1^2z_2,z_1z_2^5,z_2^6\}$. Therefore, $\mathbb{C}[z_1,z_2]/\mathfrak{a}$ is a finite $\mathbb{C}-$algebra of degree equal to \[\dim_{\mathbb{C}} \mathbb{C}[z_1, z_2]/\langle z_1^6,z_1^2z_2,z_1z_2^5,z_2^6 \rangle = 15,\] that is, for generic $r_1, r_2, b_0$ and $c$, the set $\mathcal Z$ has cardinality $15$ at most.
\end{rema}

Summarizing all of the above results, we obtain the following.

\begin{theo}
The number of limit cycles of \eqref{eq:abs} is less than or equal to $35 = 27+4+2+2$ (or $23 = 15+4+2+2$ by Remark~\ref{rema:15}).  
\end{theo}

\appendix

\section{Irreducibility of $m$}

Next we prove that, generically, the function $f(z_1,z_2)$ defined in Proposition~\ref{prop:3} is irreducible.

\begin{lema}\label{le:ceros_comunes}
Let $a_0 \neq 0.$ The following statements are equivalent.
\begin{enumerate}
\item $a(t)$ and $b(t)$ have not common zeroes in $[0,2\pi]$.
\item $(r_1+1)\left((r_1+1) b_0^2 - 2 b_0 r_2 - r_1 + 1\right) \neq 0$.
\item After the change of variables $t \to 2 \operatorname{atan}(z_1)+\pi$ and $x \to 2 \operatorname{atan}(z_2)+\pi$, the numerator of $m(t,x)$ is an irreducible polynomial over the complex. 
\end{enumerate}
\end{lema}

\begin{proof}
(1) $\Longleftrightarrow$ (2). Let us prove that $a(t)$ and $b(t)$ has common zeroes if and only if $(r_1+1)\left((r_1+1) b_0^2 - 2 b_0 r_2 - r_1 + 1\right) = 0$. Since $a(t) = a_0 (1 + r_1 \cos(t) + r_2 \sin(t))$ and $b(t) = \sin(t) + b_0 (1 - \cos(t))$, we have that 
\begin{equation}\label{ecu:sys}
\left(\begin{array}{rr} r_1 & r_2 \\ -b_0 & 1 \end{array}\right) \left(\begin{array}{c} \cos(t) \\ \sin(t) \end{array}\right) = \left(\begin{array}{c} -1 \\ -b_0 \end{array}\right).
\end{equation}

If $b_0 r_2 + r_1 = 0$, then \eqref{ecu:sys} has solutions if and only if $r_1 = -1$ and $b_0 r_2 = -1$, which is equivalent to $(r_1+1)\left((r_1+1) b_0^2 - 2 b_0 r_2 - r_1 + 1\right) = 0$ when $b_0 r_2 + r_1 = 0$. 

Otherwise, if $b_0 r_2 + r_1 \neq 0$, then
\[
\left(\begin{array}{c} \cos(t) \\ \sin(t) \end{array}\right) = \left(\begin{array}{rr} r_1 & r_2 \\ -b_0 & 1 \end{array}\right)^{-1} \left(\begin{array}{c} -1 \\ -b_0 \end{array}\right) = 
\frac{1}{b_0 r_2 + r_1}\left(\begin{array}{c} b_0 r_2 - 1 \\ b_0 (r_1 + 1) \end{array}\right),
\]
thus, the system  \eqref{ecu:sys} has solutions if and only if $(b_0 r_2 - 1)^2 + b_0 (r_1 + 1)^2 = (b_0 r_2 + r_1)^2$; equivalently, $(r_1+1)\big((r_1+1) b_0^2 - 2 b_0 r_2 - r_1 + 1 \big) = 0$.

(3) $\Longleftrightarrow$ (2).  Set $c = e^{\pi a_0}$. By the change of variables $t \to 2 \operatorname{atan}(z_1)+\pi$ and $x \to 2 \operatorname{atan}(z_2)+\pi$ we have that $m(t,x)$ becomes \[\frac{2\, f(z_1,z_2)}{(z_1^2+1)(z_2^2+1)}\] where
\begin{align*}
  f(z_1,z_2) = & -(r_1+1) z_1^2z_2 + c(r_1+1) z_1 z_2^2 + b_0(r_1+1) z_1^2 \\
               &  - 2(c-1)r_2 z_1 z_2- b_0 c (r_1 + 1) z_2^2 -(2b_0r_2+c(r_1-1)) z_1 \\
               & + (2 b_0 c r_2+r_1-1) z_2 + b_0 (c-1)(r_1-1)
\end{align*}

Suppose that $f$ factorizes into two nonconstant polynomials, say 
\[f_1 = a_1z_1^2z_2+a_2z_1z_2^2+a_3z_1^2+a_4z_1z_2+a_5z_2^2+a_6z_1+a_7z_2+a_8\]
and
\[f_2 = b_1z_1^2z_2+b_2z_1z_2^2+b_3z_1^2+b_4z_1z_2+b_5z_2^2+b_6z_1+b_7z_2+b_8.\] 
Therefore, the coefficients of the monomials of $f_1 f_2 - f = 0$ have to be equal to zero. 

Let us first notice that the coefficients of the monomials $z_1^2 z_2^2$ and $z_1^i z_2^j$ for $i > 2$ or $j > 2$ must be zero, that is, 
\begin{equation}\label{ecu:ideal1}
\begin{split}
 0 & = a_1b_1 = a_2b_1+a_1b_2 = a_2b_2 = a_4b_1+a_1b_4 \\
   & = a_3b_1+a_4b_2+a_1b_3+a_2b_4 = a_5b_1+a_3b_2+a_2b_3+a_1b_5 \\
   & = a_5b_2+a_2b_5 = a_4b_4 = a_6b_1+a_4b_3+a_3b_4+a_1b_6 \\
   & = a_7b_1+a_6b_2+a_3b_3+a_5b_4+a_4b_5+a_2b_6+a_1b_7 \\
   & = a_7b_2+a_5b_3+a_3b_5+a_2b_7 = a_5b_5 = a_6b_4+a_4b_6 = a_7b_5+a_5b_7.
\end{split}
\end{equation}
With the following Singular \cite{Singular} commands:
{\small
\begin{verbatim}
  LIB "primdec.lib";
  ring r = 0, (b0,r1,r2,c,
               a1,a2,a3,a4,a5,a6,a7,a8,
               b1,b2,b3,b4,b5,b6,b7,b8),dp;
  ideal m = 
    a1*b1,a2*b1+a1*b2,a2*b2,a4*b1+a1*b4,
    a3*b1+a4*b2+a1*b3+a2*b4,a5*b1+a3*b2+a2*b3+a1*b5,
    a5*b2+a2*b5,a4*b4,a6*b1+a4*b3+a3*b4+a1*b6,
    a7*b1+a6*b2+a3*b3+a5*b4+a4*b5+a2*b6+a1*b7,
    a7*b2+a5*b3+a3*b5+a2*b7,a5*b5,a6*b4+a4*b6,a7*b5+a5*b7;
  list L = minAss(m);
\end{verbatim}
}
\noindent we obtain that the system \eqref{ecu:ideal1} has the following eight solutions:
\[
\begin{split}
& b_7=b_6=b_5=b_4=b_3=b_2=b_1=0,\\
& b_7=b_5=b_4=b_3=b_2=b_1=a_4=a_2=a_1=0,\\
& b_6=b_5=b_4=b_3=b_2=b_1=a_5=a_2=a_1=0,\\
& b_5=b_4=b_3=b_2=b_1=a_5=a_4=a_2=a_1=0,\\
& b_5=b_4=b_2=b_1=a_5=a_4=a_3=a_2=a_1=0,\\
& b_5=b_2=b_1=a_6=a_5=a_4=a_3=a_2=a_1=0,\\
& b_4=b_2=b_1=a_7=a_5=a_4=a_3=a_2=a_1=0,\\
& a_7=a_6=a_5=a_4=a_3=a_2=a_1=0.
\end{split}
\]
Notice that the first and last solutions imply $f_2 \in \mathbb{R}$ and $f_1 \in \mathbb{R}$, respectively. So, they can be removed.

Consider now the coefficients of the monomials $z_1^i z_2^j$ of $f_1 f_2-f$ such that $0 \leq i \leq 2, 0 \leq j \leq 2$ and $(i,j) \neq (2,2)$, namely  
\begin{align*}
0 & = a_8b_1+a_6b_3+a_7b_4+a_3b_6+a_4b_7+a_1b_8+r_1+1\\
  & = -r_1c+a_8b_2+a_7b_3+a_6b_5+a_5b_6+a_3b_7+a_2b_8-c\\
  & = -r_1b_0+a_8b_4+a_6b_6+a_4b_8-b_0\\
  & = 2r_2c+a_8b_3+a_7b_6+a_6b_7+a_3b_8-2r_2\\
  & = r_1b_0c+b_0c+a_8b_5+a_7b_7+a_5b_8\\
  & = 2r_2b_0+r_1c+a_8b_6+a_6b_8-c\\
  & = -2r_2b_0c+a_8b_7+a_7b_8-r_1+1\\
  & = -r_1b_0c+r_1b_0+b_0c+a_8b_8-b_0.
\end{align*}
{\small
\begin{verbatim}
  ideal A = 
     a8*b1+a6*b3+a7*b4+a3*b6+a4*b7+a1*b8+r1+1,
     -r1*c+a8*b2+a7*b3+a6*b5+a5*b6+a3*b7+a2*b8-c,
     -r1*b0+a8*b4+a6*b6+a4*b8-b0,
     2*r2*c+a8*b3+a7*b6+a6*b7+a3*b8-2*r2,
     r1*b0*c+b0*c+a8*b5+a7*b7+a5*b8,
     2*r2*b0+r1*c+a8*b6+a6*b8-c,
     -2*r2*b0*c+a8*b7+a7*b8-r1+1,
     -r1*b0*c+r1*b0+b0*c+a8*b8-b0;
\end{verbatim}
}

If we restrict these equations to each of the eight solutions obtained above, eliminating the first and the last for the reasons given above, we obtain six systems of equations defined by the six ideals $\mathfrak a_i,\ i = 1, \ldots, 6$, of $\mathbb{C}[b_0,r_1,r_2,c,a_1, \ldots, a_7, b_1, \ldots, b_8]$. 

If we now impose the conditions $c > 0$ (actually only $c \neq -1$ is needed) and $a_0 \neq 0$, i.e. $c \neq 1$, and eliminate the variables $a_1, \ldots, a_ 8, b_1, \ldots, b_ 8$, we obtain the conditions that must be satisfied by $b_0, c, r_1$, and $r_2$, if any, for $f$ to factorize into $f_1 \cdot f_2$. 
{\small
\begin{verbatim}
  for (int i=2; i<=7; i=i+1)
  {
    eliminate(sat(A+L[i],c*(c^2-1))[1],
    a1*a2*a3*a4*a5*a6*a7*a8*b1*b2*b3*b4*b5*b6*b7*b8);
  }
\end{verbatim}
}
As a result of the above computation we conclude that equation $f_1 f_2 - f = 0$ has solutions if and only if $(r_1 + 1) \left((r_1+1) b_0^2 - 2 b_0 r_2 - r_1 + 1\right) =0$. Therefore, $f$ is irreducible if, and only if, $(r_1 +1)\big((r_1+1) b_0^2 - 2 b_0 r_2 - r_1 + 1\big) \neq 0$.
\end{proof}

\section*{Acknowledgments}

The authors thanks the referees and Renhao Tian for useful comments.

The authors are partially supported by grant number IB18023 funded by Junta de Extremadura/FEDER. The first two authors are also partially supported by grant number GR21056 funded by Junta de Extremadura/FEDER and by grant number PID2020-118726GB-I00 funded
by MCIN/AEI/10. 13039/501100011033 and “ERDF A way of making Europe”. The last author is also partially supported by grant GR21055 funded by Junta de Extremadura/FEDER and
by grant number PID2022-138906NB-C21 funded by MCIN/AEI/10.13039/501100011033 and
“ERDF A way of making Europe”

\end{document}